\documentclass[a4paper,12pt]{report}
\usepackage{amsmath}
\usepackage{amsfonts}
\usepackage{amsthm}
\usepackage{latexsym}
\usepackage{amssymb}
\usepackage{hyperref}
\usepackage{makeidx}
\usepackage{mathbbol}
\usepackage{textcomp}
\usepackage[mathscr]{euscript}

\newtheorem{thm}{Theorem}

\newtheorem{lem}{Lemma}

\newtheorem{prop}{Proposition}

\newtheorem{cor}{Corollary}

\makeindex

\hypersetup{colorlinks, citecolor=black, filecolor=black, linkcolor=black}

\begin{document}
\newcommand{\N}{\mathbb{N}}
\newcommand{\Z}{\mathbb{Z}}
\newcommand{\Q}{\mathbb{Q}}
\newcommand{\R}{\mathbb{R}}
\newcommand{\C}{\mathbb{C}}
\newcommand{\m}{\Z /m \Z}
\newcommand{\D}{\Z /d \Z}
\newcommand{\n}{\Z /n \Z}
\newcommand{\A}{\mathcal{A}}
\newcommand{\h}{\mathcal{H}}
\newcommand{\B}{\mathcal{B}}

\pagestyle{empty}

\newpage

\begin{abstract}
The study of `structure' on subsets of abelian groups, with small `doubling constant', 
has been well studied in the last fifty years, from the
time Freiman initiated the subject. In \cite{DF} Deshouillers and Freiman establish a structure theorem for subsets of $\n$ with small doubling constant. In the current article we provide an alternate proof of one of the main
theorem of \cite{DF}. Also our proof leads to slight improvement of the theorems in \cite{DF}.\\

\end{abstract}

\section{Introduction}
For a finite set $X$, by $|X|$ we will denote the number of elements in $X$. For an abelian group $G$, written additively, and subsets $A$ and $B$ of $G$ we write $A+B=\{a+b: a\in A, b\in B \}.$\\
When $G=\Z$, integers under addition, one immediately sees that $|A+B| \geq |A|+|B|-1$. The corresponding result for $\Z/ {p\Z}$,
the cyclic group with $p$ elements, was obtained by Cauchy and Davenport independantly \cite{TV}. 
They proved the following:
\begin{thm}
 For subsets $A$ and $B$ of $\Z/p \Z$, where $p$ is a prime number, one has $|A+B| \geq min\{p,|A|+|B|-1\}$.
\end{thm}
Chowla proved a `similar' theorem for any cyclic group \cite{IC}. \\

In the `Inverse problems in additive number theory' one studies the converse question, 
that is, if $|A+A|$ is not too big in comparison with $|A|$ then can we `describe' $A$? 
To be more precise, we define doubling constant of $A$ to be constant $c$ satisfying 
$|A+A|=c|A|$ \cite{TV}. Then one wants to `describe' the set $A$ when $c$ is `small'. 
We mention few well known results along this line. The following two results are for 
the additive group of integers.\\
\begin{thm}
If $A$ is a finite subset of $\Z$ with $|A+A| \leq 2|A|-1$ then $A$ is an arithmetic progression.
\end{thm}


\begin{thm}[Freiman]
Let $A$ be a subset of integers such that $|A|=k>2$. If $|A+A|=2k-1+b \leq 3k-4$, 
then $A$ is a subset of an arithmetic progression of length $k+b \leq 2k-3$.
\end{thm}

The following theorem \cite{MK}, due to Kneser, is an important result in sturcute theory of sets with small doubling.
The set $\{g\in G : g+a \in A, \forall a\in A\}$ will be called stabilizer of $A$. The theorem of Kneser is
the following:
\begin{thm}[Kneser]
 For finite subsets $A$ and $B$ of an abelian group $G$ with $|A+B|<|A|+|B|$, 
 one has $|A+B|=|A+H|+|B+H|-|H|$, where $H$ is the stabilizer of $A+B$.
\end{thm}

As a consequence of the theorem of Kneser one obtains a structure result for subsets of $\n$ with dobling constant $c<2$, see the Theorem 5.5 in \cite{TV}. 
Freiman \cite{GF} improved the result by allowing $c<2.4$, when $G$ is of prime order. 
Deshouillers and Freiman \cite{DF} extended Freiman's result for any finite cyclic group, 
albeit with a smaller $c$, greater than $2$. 
They prove a similar result with some doubling constant $c'$ when $A$ is a subset of 
$\Z \times \D$ and then they use idea of rectification (what they call `a partial lift') 
to deduce the strcture theorem for $\n$ with the doubling constant $c$.\\
In this article we give an alternate proof of their result for $\Z \times \D$, and on our way we 
make a slight improvement (Theorem 5). Also our result, Theorem 5, seems to be the best possible result one can hope for from either of the two methods. This improvement also strengthens the result for $\n$, but this we shall not discuss here.\\
The underlying idea in the proof of Theorem 5 is very simple; using Hall's marriage theorem we obtain a lower bound on the sumset $|\tilde{\B}+\tilde{\B}|$ (Proposition 6) and then we use Kneser's theorem to show that if $\tilde{\B}$ does not have the desired structure then its doubling constant is more than $2.5$ (Lemma 3 and Lemma 4).\\
In the next section we state the main theorem proved in this article. In section 3 we mention some 
known results needed for our proof.  
In section 4 we will present the proof. 

\section{Main Theorem}
In this section we mention the statement of the main theorem proved in this article. 
\begin{thm}
Let $s \geq 6$ and $d$ be positive integers. Consider integers $a_1=0,a_2,\ldots,a_s$ 
with $gcd$ of nonzero elements being $1$, and let $\B_1, \ldots, \B_s$ be subsets of 
$\D$ with $0 \in \B_1$. We let $\tilde{\B_i}=a_i \times \B_i$ and 
$\tilde{\B}=\cup_{i=1}^s \tilde{\B_i}$. Under the condition 
$$|\tilde{\B}+\tilde{\B}|<2.5|\tilde{\B}|$$ 
we have $\mbox{ max } a_i<(1.5)s$, and there exists a subgroup $\tilde{\h}$ of $\D$ 
and elements $x,y \in \D$ such that $\B_i$ is contained in the coset $a_ix+y+\tilde{\h}$ 
for each $i$. Further $|\B_j|\geq \frac{2}{3}|\tilde{\h}|$ for some $j, 1 \leq j \leq s$.\\
Moreover, 
we also have $$(max~ a_i) |\tilde{\h}|<|\tilde{\B}+\tilde{\B}|-|\tilde{\B}|.$$
\end{thm}

In Deshouillers and Freiman \cite{DF}, a corresponding theorem is proved for $s \geq 5$. 
The cases $s\leq 4$ are dealt separately. Our proof also works for smaller values of $s $ 
as well but under the condition $|\tilde{\B}+\tilde{\B}|<c_1|\tilde{\B}|$, where $c_1$ 
is a constant smaller than $2.5$ and depends on $s$. In particular, Theorem 5 is true 
with $c_1=2.4$ for $s=5$ and $c_1=2.25$ for $s=4$. We do not elaborate any more on 
the cases $s \leq 5$ and refer the reader to \cite{DF}.
In \cite{DF} the existence of the subgroup $\tilde{\h}$ is established under the hypothesis 
$|\tilde{\B}+\tilde{\B}|<\frac{5s-2}{2s+1}|\tilde{\B}|$. The assertion 
$(max~ a_i) |\tilde{\h}|<|\tilde{\B}+\tilde{\B}|-|\tilde{\B}|$, is proved in \cite{DF} under the assumption 
$|\tilde{\B}+\tilde{\B}|<2.04|\tilde{\B}|$. 


\section{Preliminaries}
In this section we develop some preliminary results needed towards the proof. \\
The well known Hall's marriage problem states the following:
\begin{thm}
Given subsets $G_1, \ldots, G_t$ contained in some set $G$,
if for every subset $I$ of $\{1,\ldots,t\}$ we have $|\cup_{i\in I}G_i| \geq |I|$ then there 
are elements $x_i \in G_i$ such that $x_i \neq x_j$, whenever $i \neq j$.
\end{thm}

We will also need following result, which is a consequence of Kneser's Theorem;
\begin{prop}
For two subsets $A \mbox{ and }B$ of a finite abelian group $G$ with $|A| \geq |B|$ 
and $|A+B|<\frac{3}{2}|A|,|B|>\frac{3}{4}|A|$ there is a subgroup $H$ of $G$ with 
$|H|<\frac{3}{2}|A|$ such that $A+B$ lies in a single coset of $H$.
\end{prop}
\begin{proof}
We have subgroup $H$ of $G$, from Kneser's theorem, satisfying;\\
\begin{equation}
|A+B|=|A+H|+|B+H|-|H|.
\end{equation}
One has $|A+H| \geq |A| > \frac{2}{3} |A+B|$ and 
$|B+H| \geq |B| > \frac{3}{4} |A| > \frac{1}{2} |A+B|$. Hence by equation (1) 
$$|H|> \frac{1}{6} |A+B|.$$\\
Now $H$ is stabilizer of $A+B$ and so $A+B$ is union of cosets of $H$, say $\mu$ cosets. 
Then one has $\mu <6$.\\
Let us assume that $A$ intersects $a$ many cosets of $H$, then 
$a|H| \geq |A| > \frac{2}{3}|A+B| \geq \frac{2\mu}{3} |H|$. This gives 
$a > \frac{ 2\mu}{3}$. Similarly if $b$ is the number of cosets of $H$ which meets $B$, 
then one has $b > \frac{\mu}{2}$. Since $|A+H|=a|H|,|B+H|=b|H|$ and $|A+B|=\mu |H|$, from equation (1)
we have $\mu|H| \geq a|H|+b|H|-|H|$. This yields a contradiction unless $\mu=1$, 
which proves the proposition.
\end{proof}
The following proposition is also a consequence of Kneser's Theorem and the proof runs along 
the same line as of proposition 1.
\begin{prop}
Consider two finite subsets $A \mbox{ and }B$ of an abelian group $G$ such that, 
$|A+B|<2|B| \mbox{ and }|B|<3/4|A|$. Then $A+B$ lies in a single coset modulo some 
subgroup $H$ with $|H|<2|B|$.
\end{prop}

The following lemma will be of help in the sequel.
\begin{lem}
Let $s \geq 2$ be an integer and $A=\{ a_1,\ldots,a_s\}$ be a subset of $[0,N-1]$ with $|A| \geq 2N/3+1$, then for any $d < |A|,$ there are elements $g_1,g_2 \in G$ such that $d=g_1-g_2$.
\end{lem}
\begin{proof}
Suppose there are no solution of $d=g_1-g_2$. Then observe that, for $a\in A$, $a+d$ is not in $A$.
If $2d \geq N$, then by considering pairs $(a,a+d)$ for $0\leq a \leq d-1$ we see that $|A| \leq d$.
This contradiction establishes the lemma for $d \geq N/2$. Now we assume $d<N/2$. Let $N=2dr+q, 0\leq q <2d$. Now any interval $[m,m+2d) \subset [0,s-1]$ can have at most $d$ elements in $A$. 
Hence by breaking the interval $[0,N-1]$ into sub intervals of length $2d$ and one sub interval of 
length less than $2d$, we see that $|A| \leq [\frac{N}{2d}]d+\mbox{min}\{q,d\}=rd+\mbox{min}\{q,d\}$.
When $q<d$ then we see that $|A| \leq rd+q <2N/3$, which is a contradiction. When $d\leq q$ then $|A| \leq rd+d \leq 2N/3$. Thus we get $|A|\leq 2N/3$. This contradiction establishes the lemma.
\end{proof}
For any subset $G$ of $[0,N-1]$, we will define $G^{(1)}$ to be the set of those elements $d$ of $[0,N-1]$ 
which satisfy a relation of the form $d=b+c-a$ for $a,b,c \in G$, not necessarily distinct. $G^{(i+1)}$ 
is defined from $G^{(i)}$ inductively. We will define $G^{good}=\cup_{i \geq 0} G^{(i)}$, 
with $G^{(0)}=G$. In case we have a subset $A \subset [0,N-1]$ at hand and $G \subset A$, 
then $G^{good}$ shall be obtained by intersecting $G^{(i)}$ with $A$ at each step. 
The phrase `$A$ misses an element $a$' will be used in the sense that $a$ is not in $A$. 
We have following useful proposition.
\begin{prop}
For an integer $s> 3$, consider $A \subset [0,s-1]$ with $|A| \geq 2s/3 +1$ then we can choose a set $G$ of two elements $a_i,a_j$ from $A$ such that $G^{good}=A$ and $a_j-a_i=1$.
\end{prop}
\begin{proof}
We use induction on $s$. Let $[s/2]$, be the integral part of $s/2$. Then for an appropriate $t \in \{[s/2]-1,[s/2],[s/2]+1\}$ either the set $A'=A \cap [0,t-1]$ or the set $A'=A \cap [t,s-1]$ satisfies the hypothesis in the proposition (in the later case after a translation by $t$). We chose $t$ such that cardinality of $A'$ is the largest possible. Let us consider the first case, Now by induction hypothesis we can chose a $G'$ such that $G'^{good} \cap A'=A'$ (in second case we translate by $t$). Now we see that the distance between maximum from $A'$ and minimum of $A \cap [t,s-1]$ is not more than $s/3$, hence by Lemma 1, this difference is also achieved as difference of two elements from $G'^{good}$. This shows that minimum of $A \cap [t,s-1]$ is in $G'^{good}$. Similar reasoning shows that $G'^{good} \cap A=A$. Take $G=G'$, this proves the proposition.
\end{proof}

The following result can be found in \cite{JS}.
\begin{prop}
Let $\mathscr{U} \mbox{ and }\mathscr{V}$ be to non empty set of integers such that 
$$\mathscr{V}=\{v_1<\ldots<v_t\}\subset \mathscr{U}=\{0=u_1<\ldots<u_s\}$$ and $gcd(u_2,\ldots,u_s)=1$. We have$|\mathscr{U}+\mathscr{V}| \geq min(u_s+t, s+2t-3)$;\\
moreover, if $\mathscr{U} \neq \mathscr{V}$ and $u_s = s+t-2$, then $|\mathscr{U}+\mathscr{V}| \geq u_s+t$.
\end{prop}
Now onwards we will use the same notations as in the Theorem 5 and put $A'=\{a_1,\ldots,a_s\}$. We shall define $R=min\{max \mbox{ }a_i-s+3, s\}$. Clearly $2 \leq R \leq s$. We have the following;
\begin{lem}
$|A'+A'| \geq 2s+R-3$.
\end{lem}

\begin{proof}
We will consider sets\\ $G_{1,1}=\ldots=G_{1,s-1}=a_1+A',$\\
$G_{2,1}=G_{2,2}=a_2+A',G_{3,1}=G_{3,2}=a_3+A',\ldots,G_{R,1}=G_{R,2}=a_R+A',$ \\
$G_{R+1}=a_{R+1}+A',\ldots G_{s}=a_s+A'$.\\
Using the proposition 4 it is easy to verify that the conditions of Hall's marriage problem are satisfied for the family $G_{ij}$, which yields
 $|A'+A'| \geq (s-1)+2(R-1)+(s-R)=2s+R-3$.
\end{proof}
The proof of Lemma 2 also shows that there are distinct $s-1$ elements in $A'+A'$ with $a_1$ as a summand, $2$ elements with $a_i$ as a summand, for each $2 \leq i \leq R$ and one elements with $a_i$, for $i>R$, as a summand. Now we proceed to obtain a refinement of Lemma 2.\\
Let $a$ be the largest integer such that $A'$ misses $a$ elements from the interval $[0,2a-1]$ (in case there are no $a$ satisfying this then we take $a=0$) and let $b$ be the largest integer such that $A'$ misses $b$ elements from $[s+R-2b-2,s+R-3]$. Finally let $c$ be the number of elements $A'$ misses from $[2a,s+R-2b-3]$. We have the following;\\
\begin{prop}
 $|A'+A'| \geq 2s+R-3+c$.
\end{prop}
\begin{proof}
To prove the proposition, we make the following observation;\\
for any $n \leq s+R-3$, if number of elements $A'$ misses from $[0,n]$ is strictly less than $\frac{n+1}{2}$, then $n \in A'+A'$.\\
Using this we conclude that every element of the interval $[2a,s+R-3]$ is in $A'+A'$. Considering $\{s+R-3-a_i:a_i \in A'\},$ we see that every element of the interval $[s+R-3+2b,2(s+R-3)]$ appears in $A'+A'$. Also there are $a$ elements from $[0,2a-1]$ and $b$ elements from $[s+R-3,s+R-3+2b]$ appearing in $A'+A'$. But from the $b$ elements of $[s+R-3,s+R-3+2b]$, the element $s+R-3$ is already considered. Hence $|A'+A'|\geq [(s+R-3)-2a+1]+[2(s+R-3)-(s+R-3+2b)+1]+a+b=2s+R-3+c$, as $a+b+c=R-2$.
\end{proof}

\section{Proof of Theorem 5}

We will assume the setup as in Theorem 5. By lemma 2 we have $|\Pi_1(\tilde{\B}+\tilde{\B})| \geq 2s+R-3$, where $ \Pi_1(\tilde{\B}+\tilde{\B})$ denotes the first co-ordinate of $ \tilde{\B}+\tilde{\B}$. By considering the second coordinates, we give a lower bound on $|\tilde{\B}+\tilde{\B}|.$
\bigskip
\bigskip

\begin{prop}
We have the following lower bound,
\begin{eqnarray}
|\tilde{\B}+\tilde{\B}| &\geq & \sum_{j=1}^{s-1}|\B_1+\B_{1,j}|+\sum_{j=s}^{s+1}|\B_2+\B_{2,j}| +\ldots+ \sum_{j=s+2R-4}^{s+2R-3}|\B_R+\B_{R,j}| \cr
&& +|\B_{R+1}+\B_{R+1,j}|+\ldots+|\B_s+\B_{s,j}| 
\end{eqnarray}
where $\B_{i,j} \in \{\B_1, \ldots, \B_s \}$ and for a fixed $i$ the $\B_{i,j}'s$ are distinct.
\end{prop}

\begin{proof}
In the proof of Lemma 2 we have produced $2s+R-3$ elements in $\Pi_1(\tilde{\B}+\tilde{\B})$.
There are $s-1$ elements of the form $a_1+a_j$, $2$ elements of the form $a_i+a_j$ for $i=2,\ldots, R$, and $1$ element of the form $a_i+a_j \mbox{ for }i>R$.\\
\bigskip

This gives us;
\begin{eqnarray*}
 &&\{(a_1+a_i,\B_1+\B_i) \mbox{ for some } s-1 \mbox{ values of }i, 1 \leq i \leq s\} \cup \cr
&& \{(a_2+a_i,\B_2+\B_i) \mbox{ for some } 2 \mbox{ values of }i, 1 \leq i \leq s\} \cup \cr
\Huge{\tilde{\B}+\tilde{\B} \supset} &&\ldots \cup \{(a_R+a_i,\B_R+\B_i) \mbox{ for some } 2 \mbox{ values of }i, 1 \leq i \leq s\} \cup \cr
&&\{(a_{R+1}+a_i,\B_{R+1}+\B_i) \mbox{ for some } 1 \leq i \leq s\} \cup \ldots \cup \cr
&&\{(a_s+a_i,\B_s+\B_i) \mbox{ for some } 1 \leq i \leq s\} \cr
\end{eqnarray*}
This way of listing elements of $\tilde{\B}+\tilde{\B}$ plays a critical role in the proof. As a consequence the proposition follows.
\end{proof}
Since $|\B_i+\B_j| \geq max\{|\B_i|,|\B_j|\}$, we obtain the following,
\begin{cor}
$|\tilde{\B}+\tilde{\B}|-|\tilde{\B}| \geq (s-2)|\B_1|+|\B_2| +\ldots+ |\B_R|.$
\end{cor}
Mostly we will be assuming $|\B_1| \geq \ldots \geq |\B_s|$ so that the lower bound obtained in the corolarry 1 is the best by this method, but at times we will be assuming different ordering too.
Even though Proposition 6 holds for all values of $R$, for $R=2,3$ we will need different consideration at times, which are little stronger. For $R=3$ we see that $A'=\{0,1, \ldots, s\}$ with one element $ i_0 \neq 0 \mbox{ omitted }$. Here we claim that;
 \begin{eqnarray*}
 &&\{(a_1+a_i,\B_1+\B_i) \mbox{ for some } s \mbox{ values of }i, 1 \leq i \leq s\} \cup\\
&& \{(a_2+a_i,\B_2+\B_i) \mbox{ for some } 2 \mbox{ values of }i, 1 \leq i \leq s\} \cup \\
\Huge{\tilde{\B}+\tilde{\B} \supset} &&\cup \{(a_3+a_i,\B_3+\B_i) \mbox{ for some }  \mbox{ values of }i, 1 \leq i \leq s\} \cup\\
&&\{(a_4+a_i,\B_4+\B_i) \mbox{ for some } 1 \leq i \leq s\} \cup \ldots \cup \\
&&\{(a_s+a_i,\B_s+\B_i) \mbox{ for some } 1 \leq i \leq s\}.
\end{eqnarray*}
This is achieved by considering the family of sets \\
$G_{1,1}= \ldots =G_{1,s}=a_1+A',$\\
$G_{2,1}=G_{2,2}=a_2+A',G_3=a_3+A',\ldots,G_{s}=a_s+A',$ \\
and then applying the Hall's marriage theorem.

This immediately yields,
\begin{equation}
|\tilde{\B}+\tilde{\B}| \geq \sum_{j=1}^s|\B_1+\B_{1,j}|+\sum_{j=s+1}^{s+2}|\B_2+\B_{2,j}|+|\B_3+\B_{3,j}|+\ldots+|\B_s+\B_{s,j}|.
\end{equation}
For $R=2$ a similar consideration yields,
\begin{equation}
|\tilde{\B}+\tilde{\B}| \geq \sum_{j=1}^s|\B_1+\B_{1,j}|+|\B_2+\B_{2,j}|+\ldots+|\B_s+\B_{s,j}|.
\end{equation}

\begin{prop}
Under the assumption of Theorem 5 one has $max \mbox{ } a_i < 1.5s$.
\end{prop}
\begin{proof}
We shall give the proof under the assumption $R\geq 4$, the cases $R=2,3$ can be worked out similarly using the equations (3) and (4) in place of equation (2). Using the trivial lower bound on $|\B_i+\B_j|$ in equation (2) we obtain,
$$|\tilde{\B}+\tilde{\B}| \geq (s-1)|\B_1|+2|\B_2| +\ldots+ 2|\B_R|+|\B_{R+1}|+\ldots+|\B_s|.$$
Since $|\tilde{A}|=\sum_i|\B_i|$, from above we obtain\\
$$|\tilde{\B}+\tilde{\B}|-|\tilde{\B}|\geq \left( (s-2)|\B_1|+|\B_2|+\ldots+|\B_R|\right).\frac{s+R-3}{s+R-3}.$$
Now $$\frac{(s-2)|\B_1|+|\B_2|+\ldots+|\B_{R-1}|}{s+R-3},$$ being average of $|\B_1|,\mbox{ }s-2$ times, and $|\B_2|,\ldots,|\B_{R-1}|$, is greater than the average of $|\B_1| , \ldots, |\B_s|$, namely $$\frac{\sum_i|\B_i| }{s}$$ (as $|\B_1| \geq |\B_i| \mbox{ and } s+R-3 \geq s$). This gives $$|\tilde{\B}+\tilde{\B}|-|\tilde{\B}|\geq \frac{\sum_i|\B_i| }{s}(s+R-3).$$ Now because of the assumption that $|\tilde{\B}+\tilde{\B}|-|\tilde{\B}|\leq 1.5 \sum_i|\B_i| $, we get $s+R-3 < 1.5s$. Thus $R\neq s$ and hence $R=max \mbox{ } a_i-s+3$, and this gives the proposition.
\end{proof}

Now we intend to exhibit the subgroup $\tilde{\mathscr{H}}$ as sought in Theorem 5. This will be done through next few lemmas. 

\begin{lem}
There exists a subgroup $\tilde{\mathscr{H}}$ of $\D$ such that $\B_1$ lies in a single coset of $\tilde{\mathscr{H}}$ and $|\tilde{\mathscr{H}}|< \frac{3}{2} |B_1|$.
\end{lem}
\begin{proof}
If $ |\B_1+\B_i| < \frac{3}{2} |\B_1|$ and $|\B_i| > 3/4 |\B_1|$, then from Proposition 1 we get that $\B_1+\B_i$ lies in a single coset of the stabilizer and consequently $\B_1$ also lies in a single coset. We will partition $\B_i's$ in three different categories. \\
\begin{eqnarray*}
U=& \{a_i:|\B_i| \geq 3/4|\B_1|, \B_i \mbox{ does not lie in a single coset of any subgroup } \\
&\tilde{\h} \mbox{ of cardinality not bigger than } \frac{3}{2}|\B_i| \},\\
V=& \{a_i:|\B_i| \geq 3/4|\B_1|, \B_i \mbox{ lies in a single coset of some subgroup } \\
& \tilde{\h} \mbox{ of cardinality not bigger than } \frac{3}{2}|\B_i| \} \cup \{a_i:\frac{1}{2} |\B_1| \leq |\B_i| < 3/4|\B_1| \}
\end{eqnarray*}
$W=\{a_i: |\B_i| \leq \frac{1}{2} |\B_1| \}$.\\
We let $u,v \mbox{ and }w$ denote the respective cardinality. \\
If there is no subgroup as claimed in the lemma, then as seen in the Proposition 1 and Proposition 2, one has\\
$~~~~ |\B_1+\B_i| \geq \frac{3}{2} |\B_1|$, ~~~~\quad if $a_i \in U$,\\
$~~~~ |\B_1+\B_i| \geq 2 |\B_i|$, ~~~~~~~\quad if $a_i \in V$,\\
\quad $~~~~|\B_1+\B_i| \geq |\B_1|$ ~~~~~~~~~\quad otherwise.\\

Now we will assume that $R \geq 4$. Then from equation (2) one obtains\\
$$|\tilde{\B}+\tilde{\B}| \geq \frac{3(u-1)}{2} |\B_1|+2 \sum_{a_i \in V}| \B_i|+w|\B_1|+2(|\B_2|+\ldots+|\B_R|)+|\B_{R+1}|+\ldots+|\B_s|.$$ This gives\\
$$|\tilde{\B}+\tilde{\B}|-\tilde{\B} \geq \frac{3}{2} (u-1)|\B_1|+2 \sum_i |\B_i|+w|\B_1|-|\B_1|+(|\B_2|+\ldots+|\B_R).$$ But we are given $|\tilde{\B}+\tilde{\B}|-|\tilde{\B}|<1.5|\tilde{\B}|$. Comparing the two inequalities we get,\\
$$1.5|\tilde{\B}| > \frac{3}{2} (u-1)|\B_1|+2 \sum_{a_i \in V} |\B_i|+w|\B_1|-|\B_1|+(|\B_2|+\ldots+|\B_R|).$$ 
From this we obtain
\begin{equation}
1.5 \sum_{a_i\in U \cup W}|\B_i|> \frac{3}{2} (u-1)|\B_1|+0.5 \sum_{a_i \in V} |\B_i|+w|\B_1|-|\B_1|+(|\B_2|+\ldots+|\B_R|).
\end{equation}
Let us put $|\B_i^{'}|= |\B_1|$ if $a_i \in U$, $|\B_i^{'}|=\frac{1}{2} |\B_1|$ if $a_i \in V$ and $|\B_i^{'}|=\frac{1}{2} |\B_1|$ in case $a_i \in W$. Since the coefficients of $\B_j$ for $j \geq 2$ is more on left side than right side for $j \in U \cup W$ and more on right side for $j \in V$ we can replace $\B_i$ by $\B_i'$ to obtain 
$$0>\frac{-5}{2}|\B_1|+\frac{1}{4}(v+w)|\B_1|+|\B_2^{'}|+|\B_3^{'}|+|\B_4^{'}|,$$ substituting the values of $\B_2^{'},\B_3^{'},\B_4^{'}$ this yields a contradiction as $s=u+v+w>5$. For $R=3$ we use  equation (3), similarly, to arrive at a contradiction.

We give the sketch of the proof for $R=2$. Again we see that if $v+w>3$ then a contradiction can be derived as mentioned above. So we assume $v+w \leq 3$. The proof mentioned below is for $v+w=3$ but similar arguments work for other cases too.
Here we break $A'+A'$ in $U+U,U+(V\cup W)$ and $(V\cup W) +(V \cup W)$. We note that in $A'+A'$, we can consider 3 elements from $V\cup W +V \cup W$, 3 elements from $U+(V\cup W)$ and rest from $U+U$. For $U+U$ we use Hall's marriage problem, as in Lemma 2. Using Proposition 1 we get at least
$$\frac{3}{2} u|\B_1|+\frac{3}{2} |\B_2|+\ldots +\frac{3}{2} |\B_u|$$ elements in $\tilde{\B}+\tilde{\B}$ with first co-ordinate in $U+U$. Using Proposition 1 and Proposition 2 we get atleast max$\{ |\B_u|,2|\B_{u+1}|\}+$max$\{|B_u|,2|\B_{u+2}|+$max $\{|B_u|,2|\B_{u+3}| \}$ more elements in $\tilde{\B}+\tilde{\B}$ with first co-ordinate in $U+(V \cup W)$. Also we get at least $2|\B_{u+2}|+|\B_{u+3}|$ more elements in $\tilde{\B}+\tilde{\B}$ with first co-ordinate in $(V \cup W) +( V \cup W)$. Clearly max$\{ |\B_u|,2|\B_{u+1}|\} \geq 2|\B_{u+1}|$, max$\{ |\B_u|,2|\B_{u+2}|\} \geq 3/4|\B_u|+\frac{1}{2}|\B_{u+2}|$ and max$\{ |\B_u|,2|\B_{u+3}|\} \geq 1/4 |\B_u|+\frac{3}{2}|\B_{u+3}|$. This gives us
\begin{equation}
\frac{3}{2} u|\B_1|+\frac{3}{2} |\B_2|+\ldots +\frac{3}{2}|\B_{u-1}|+\frac{5}{2} |\B_u|+2|\B_{u+1}|+\frac{5}{2}|\B_{u+2}|+\frac{5}{2}|\B_{u+3}|<2.5|\tilde{\B}|.
\end{equation}
Equation (6) gives
$$\frac{3}{2}(u-1)|\B_1|<(|\B_1|+\ldots +|\B_{u-1}|)+\frac{1}{2}|\B_{u+1}|.$$
This gives a contradiction as $u\geq 3$.
\end{proof}
Next we shall show that each of the $\B_i$ lies in a single coset of $\tilde{\mathscr{H}}$ for some subgroup $\tilde{\mathscr{H}}$ of $\D$ with $|\tilde{\mathscr{H}}|<\frac{3}{2} |\B_1|$. This is content of Lemma 4.
\begin{lem}
There is a subgroup $\tilde{\mathscr{H}}$ of $\D$ with $|\tilde{\mathscr{H}}|<\frac{3}{2} |\B_1|$ such that each of the $\B_i$ is contained in a single coset of $\tilde{\mathscr{H}}$.
\end{lem}
\begin{proof}
We shall give the proof when $R \geq 4$. The cases $R=2,3$ can be handled with a bit more careful working. Here we shall consider $\tilde{\B}$ in various different ordering. Let us write \\
\begin{eqnarray*}
 \tilde{\B} &=&\{ (a_i,C_i):1\leq i \leq r;|C_1| \geq |C_2| \geq \ldots \} \cup \\
 &&\{ (a_{r+j},D_j):1\leq j \leq t;|D_1| \geq |D_2| \geq \ldots \},
 \end{eqnarray*}
where $C_i \mbox{ lies in a singel coset modulo }H \mbox{ for the subgroup  } H \mbox{ of lemma 3 } $ and $D_j \mbox{ does not lie in a singel coset modulo }H $.\\
By Lemma 3, we have $C_1=\B_1$ and one immediately has $|C_1+C_i| \geq |C_1|$ and $|C_1+D_j| \geq 2 |C_1|$. Now using the description of $\tilde{\B}+\tilde{\B}$, we obtain;\\
$$|\tilde{\B}+\tilde{\B}| \geq (r+2t-2)|\B_1|+2|\B_2|+\ldots+2|\B_R|+|\B_{R+1}+\ldots+|\B_s|.$$
Next let us write (after a rearrangement),\\
$\tilde{\B} = \{ (a_i,\B_i): \B_i=D_i \mbox{ for }1\leq i \leq t \mbox{ and }B_{t+i}=C_i \mbox{ for }1 \leq i \leq r \}$.\\
Proceeding in the same way for this listing of $\tilde{\B}$, as we had done to obtain equation (2), we get 
\begin{eqnarray*}
|\tilde{\B}+\tilde{\B}| & \geq & t|D_1|+2|C_2|+\ldots+2|C_r|+2|D_2|+\ldots+2|D_4|+|D_5|+\ldots+|D_t| \\
&& +|C_1|+\ldots+|C_r|,
\end{eqnarray*}
assuming $R \geq 4 $ and $t \geq 4$. When $t \leq 3$ then also the method works with appropriate changes. We add the two lower bounds on $|\tilde{\B}+\tilde{\B}|$ to obtain,
\begin{eqnarray*}
&& (r+2t-2)|\B_1|+2|\B_2|+\ldots+2|\B_R|+|\B_{R+1}+\ldots+|\B_s| \\
5 |\tilde{\B}| &>& +t|D_1|+2|C_2|+\ldots+2|C_r|+2|D_2|+\ldots+2|D_4|+|D_5|+\ldots+|D_t| \\
&& +|C_1|+\ldots+|C_r|.
\end{eqnarray*}
We note that $$|\tilde{\B}|=\sum_{i=1}^s |\B_i|=\sum_i|A_i|=\sum_{i=1}^r|C_i|+\sum_{i=1}^t|D_i|.$$
Also $|\B_1|=|C_1|$, this immediately yields,
$$3 |\tilde{\B}| >(r+2t-3)|\B_1|+|\B_2|+\ldots+|\B_R| +(t-1)|D_1|+2|C_2|+\ldots+2|C_r|+|D_2|+\ldots+|D_4|.$$
Further we have $s=r+t$ and $(s-3)|\B_1|+|\B_2|+\ldots+|\B_R|\geq |\tilde{\B}|$, so we get,
$$ 2|\tilde{\B}| >t|\B_1|+(t-1)|D_1|+2|C_2|+\ldots+2|C_r|+|D_2|+\ldots+|D_4|.$$
Since $(t-2)|\B_1|+|D_3|+|D_4| \geq \sum_i |D_i| \mbox{ and }(t-1)|D_1|+|D_2| \geq \sum_i |D_i|$, the above yields $2|\tilde{\B}|>2|\tilde{\B}|$, a contradiction.
\end{proof}
Let $x_i  \in \D$ be such that $\B_i \subset x_i+\tilde{\mathscr{H}}$. Next we shall prove the existence of $x,y \in \D$ satisfying $\B_i \subset a_i x+y+\tilde{\mathscr{H}}$. We shall give the proof for $R\geq 4$. The basic idea is to show that if such an $x$ and $y$ can not be obtained then it will result in more terms on right side of equation (2), which will exceed the limit. This is made precise below.\\
We see that for integers $a_i$ and $x_i$ the following holds:
$$\mbox{Claim: } a_i-a_j=a_r-a_s \Longrightarrow x_i-x_j=x_r-x_s.$$


Now we proceed to establish the claim. First we note that, from Lemma 2, $R < s/2+3$ and $max \{a_i\} < 1.5 s$. For $1 \leq k  \leq R-2$ we will consider,\\
$$S_k=\{(a_i,a_j) \in A' \times A':a_j-a_i=k \}.$$
Note that $|S_k|\geq s-R-k+1$. To see this we form pairs $(a_i,a_j)$ with all choices of $0 \leq a_i,a_j \leq max \{a_i\}$ satisfying $a_j-a_i=k$, there are exactly $s+R-3-k$ many such pairs. Of these, at most $2(R-2)$ of them can have either $a_i$ or $a_j$ not in $A'$. Now for $(a_i,a_j)\mbox{ and }(a_u,a_v) \in S_k$ we will define $(a_i,a_j)\backsim(a_u,a_v)$ if $x_j-x_i=x_v-x_u \pmod {\tilde{\mathscr{H}}}$. We contend that under this equivalence $S_1$ has only one equivalence class. Let $S_1=\sqcup_{j=1}^t S_{1j}$ be the decomposition of $S_1$ in disjoint equivalence classes. We want to show that $t=1$. Let us assume $t>1$. For $j \neq j'$ and $(a_u,a_v) \in S_{1j},(a_w,a_z)\in S_{1j'}$ we have $a_v+a_w=a_u+a_z$, where as $\B_v+\B_w \cap \B_u+\B_z=\emptyset$. For the lower bound on $|2\tilde{\B}|$ in equation (2) we had considered at most one of $\B_v+\B_w \mbox{ and } \B_u+\B_z=\emptyset$ corresponding to the first co-ordinate $a_v+a_w=a_u+a_z$. This reasoning shows that, if $S_1$ has more than one equivalence class then we can improve upon equation (2) to obtain a better lower bound.\\
Let $S_{1j_0}$ have maximum cardinality among all equivalence classes, $j_0$ need not be unique and if there are more choices we fix any one. We arrange elements of $S_{1j_0}$ with first co-ordinate in increasing order, and consider them as a row. Also we can arrange the elements of $S_1$ which are not in $S_{1j_0}$ with first co-ordinate in increasing order, and consider them as a column. For every $(a_u,a_v)$ in the row and every $(a_w,a_z)$ in the column we have an element $a_v+a_w=a_u+a_z$ in $A'+A'$, and there are at least $|S_1|-1$ distinct such elements. Corresponding to the first co-ordinate $a_v+a_w=a_u+a_z$, at most one of $\B_v+\B_w \mbox{ and } \B_u+\B_z=\emptyset$ was considered in equation (2). Thus we get at least $s-R-1$ many more summands in equation (2).
Note that these summands are different because of the condition $\B_u+\B_z \cap \B_v+\B_w=\emptyset$. We obtain,
\begin{eqnarray*}
|\tilde{\B}+\tilde{\B}| &\geq & \sum_{j=1}^{s-1}|\B_1+\B_{1,j}|+\sum_{j=s}^{s+1}|\B_2+\B_{2,j}| +\ldots+ \sum_{j=s+2R-4}^{s+2R-3}|\B_R+\B_{R,j}| \\
&& +|\B_{R+1}+\B_{R+1,j}|+\ldots+|\B_s+\B_{s,j}| \\
&&+(|\B_s |+\ldots +|\B_{s-(s-R-2)}|).
\end{eqnarray*}
This lower bound leads to a contradiction, by noticing that each $|\B_i|$ can be replaced by $|\B_1|$ and there are at least $\frac{5s}{2}$ terms on the right side.
This proves that $S_1$ shall have only one equivalence class. Similar analysis shows that $S_k$ has only one equivalence class. \\
Thus, we have $$a_j-a_i=a_r-a_s \Longrightarrow x_j-x_i=x_r-x_s.$$ An application of the following theorem, which is an important result in itself, establishes the existence of $x$ and $y$ as claimed. Though the notations are same as in the Theorem 5, but we make the statement indpendant of previous notations, as this result might be of importance at some other places too.
\begin{thm}
Let $s \geq 6$ and $A=\{0=a_1, \ldots, a_s \} \subset [0,N-1]$, with $gcd(a_2, \ldots, a_s)=1$. Also consider integers $x_1, \ldots, x_s$ satisfying $$a_j-a_i=a_r-a_s \Longrightarrow x_j-x_i=x_r-x_s.$$  If $|A+A|<\frac{5}{2}s$, then there exist $x,y$ such that $x_i=a_ix+y$ for each $i$.
\end{thm}

Our strategy to prove Theorem 7 is to produce two elements $a_i,a_j \in A$ with $a_j-a_i=1$ such that for $G=\{a_i,a_j\}$ we have $G^{good}=A$. Once we have two elements $a_i,a_j$ satisfying this, then we can solve for $x,y$ satisfying $$x_i=a_ix+y \mbox{ and }x_j=a_jx+y.$$ By the definition of $G^{good}$ it is clear that for every $x_k \in G^{good}$ we have $x_k=a_kx+y$.\\
We shall give the proof for $s \geq 17$, for the smaller values of $s$ it is possible to check case by case and we omit the proof.
We can assume $N=s+R-2$. By Lemma 2 we have $R<\frac{s}{2}+3$. Also if $s \geq \frac{2}{3}(s+R-2)+1$, i.e. $2R \leq s+1$, then by Proposition 3 we succeed in obtaining $a_i,a_j$ as needed. So we assume  $s+2 \leq 2R \leq s+5.$ Let $c$ be as in the proposition 6. Since $|A+A|<\frac{5}{2}s$, we obtain $c \leq 1$. We discuss the case $c=1$, the case $c=0$ being similar. $A$ misses one element from the interval $I=[2a,s+R-2b-3]$, hence Proposition 3 is applicable for the set $A'=A \cap I$, provided $2s-12 \geq 2R.$ Since $s \geq 17$ we can chose two elements $a_i,a_j \in A'$ with $a_j-a_i=1$ and for $G'=\{a_i,a_j\}$ we have $G'^{good}=A'.$ Even though $A'$ misses one element from $I$ but the set $A'+A'-A'$ coincides with the set $I+I-I$. The latter is $[4a-s-R+2b+3,2s+R-2a-4b-6]$. Now we consider two cases $a \leq b$ and $a>b$.\\
case 1- $a \leq b$.\\
Let $k$ be an integer such that $G'^{good}=G'^{(k)}$, then from above it is clear that for $G=G' \in A$, we have $G^{(k+1)}=a \cap [0,s+R-2b-3]$. Also since any element of $A \cap [s+R-2b-3, s+R-3]$ is at most at a distance of $R-2-a-1$ from $s+R-2b-3$, hence using arguments as in Lemma 2 show that $G^{(k+2)}=A$.\\
case 2- $a>b$.\\
Here the proof is similar, so we do not give the details.
This completes the proof of Theorem 7.

Next we intend to prove the last assertion of the Theorem 5, namely; 
\begin{equation}
max~ a_i |\tilde{\h}|<|\tilde{\B}+\tilde{\B}|-|\tilde{\B}|.
\end{equation}
Note that $max~a_i=s+R-3$. We consider sets,\\
$U=\{a_i:|\B_i| \geq \frac{2}{3}|\tilde{\mathscr{H}}| \},\quad V=\{a_i:\frac{1}{3}|\tilde{\mathscr{H}}| \leq |\B_i| < \frac{2}{3}|\tilde{\mathscr{H}}| \},$ \\ $W=\{a_i:|\B_i| < \frac{1}{3}|\tilde{\mathscr{H}}| \}$. We shall write $u,v,w$ for the number of elements in $U,V,W$ respectively, also elements of $U$ will be denoted by $u_1< \ldots < u_u$.\\
A first coordinate of $\tilde{\B}+\tilde{\B}$ will be referred as `good' if it can be represented in the form $u+v \mbox{ with }u \in U, v\in U \cup V$. Every `good' first coordinate contributes $|\tilde{\mathscr{H}}|$ in equation (2). Showing that $w \leq 1$, establishes equation (7) with the help of equation (2). 
\begin{lem}
Under the hypothesis of the Theorem 5 we have $u \geq w+2R-3$.
\end{lem}
\begin{proof}
First we establish $u \geq R$.\\
On the contrary if $R>u$, then equation (2) gives us the inequality\\
$$(u+v-1)|\tilde{\mathscr{H}}|+(w+u-2)|\B_1|+2/3(R-u)|\tilde{\mathscr{H}}| <1.5 u|\B_1|+v|\tilde{\mathscr{H}}|+.75 w |\B_1|.$$
This is proved by shifting $\B_2, \ldots \B_R$ on the right side of equation (2) and noticing that the coefficient is positive so one can replace them by a bigger quantity, namely $|\B_1|$. Simplifying this yields 
$$u |\tilde{\mathscr{H}}|+0.25 w |\B_1|+2/3(R-u)|\tilde{\mathscr{H}}|<(0.5 u+2) |\B_1|+|\tilde{\mathscr{H}}|,$$
Since $R>u, u\geq 1, w\geq 2$ we obtain a contradiction.
Thus we have $u \geq R$. Using this in equation (2) yields,
$$(u+v-1)|\tilde{\mathscr{H}}|+(w+R-2)|\B_1|<1.5 u|\B_1|+v|\tilde{\mathscr{H}}|+0.5 w |\tilde{\mathscr{H}}|.$$
As $2/3|\tilde{\mathscr{H}}| \leq |\B_1| \leq |\tilde{\mathscr{H}}|$, the above inequality shall be true for one of the extreme value of $|\B_1|$, as the inequality is linear in $|\B_1|$. Putting $|\B_1|=2/3 |\tilde{\mathscr{H}}|$ gives a contradiction and $|\B_1|=|\tilde{\mathscr{H}}|$ gives $u \geq w+2R-5$. We need to gain a bit more. Since $R \geq 4$, we have $u \geq R+1$. Let us partition $W$ in two parts,\\
$W_1=\{a_i:|\B_i| \geq |\tilde{\mathscr{H}}|-|\B_R| \}, \quad W_2=\{a_i:|\B_i| < |\tilde{\mathscr{H}}|-|\B_R| \}$ with $w_1,w_2$ being their cardinality respectively. Again if $w_2 \leq 1$ then equation (2) will already prove the assertion of the Theorem. As in this case in equation (2) there will be at most $s$ first coordinates with summands from $W_2$ which will contribute at least $|\tilde{\B}|$ in equation (2) and rest will contribute $(s+R-3)|\tilde{\mathscr{H}}|$. Equation (2) helps us in having,
\begin{align*}
\left\{\begin{matrix}
(u+v+w_1-1)|\tilde{\mathscr{H}}|+ \\
 (w_2+R-3)|\B_1|+|\B_R|  
 \end{matrix}\right\}
  & < 
\left\{\begin{matrix}  
  1.5(R-1)|\B_1|+1.5(u-R+1)|\B_r|+ \\
 v|\tilde{\mathscr{H}}| +0.5w_1|\tilde{\mathscr{H}}|+1.5w_2(|\tilde{\mathscr{H}}|-|\B_R|)
 \end{matrix}\right\},
\end{align*}
i.e. $$[u+0.5w_1-1.5w_2-1]|\tilde{\mathscr{H}}|+[w_2-0.5R-1.5]|\B_1|+[1.5(w_2-u+R)-0.5]|\B_R|<0.$$
Since the last inequality does not hold for $|\B_1|=\frac{2}{3}|\tilde{\mathscr{H}}|$ (which in turn will also give $|\B_1|=|\B_R|$), so it shall be true when $|\B_1|$ is replaced by $|\tilde{\mathscr{H}}|$. This gives us,
$$(u+0.5w_1-0.5w_2-0.5R-2.5)|\tilde{\mathscr{H}}|+(1.5w_2-1.5u+1.5R-0.5)|\B_R|<0.$$
Again, as $\frac{2}{3}|\tilde{\mathscr{H}}| \leq |\B_R| \leq |\tilde{\mathscr{H}}|$, as done earlier, we obtain $u>w_1+2w_2+2R-6$. As $w_2 \geq 2$ we get $u\geq w+2R-3$, this proves the lemma.
\end{proof}

We call an element of $A'$ `desirable' if all the first coordinates (of $\tilde{\B}+\tilde{B}$) it contributes to are `good' and we say it is `almost desirable' if all but one coordinates it contributes to are `good'. Our aim is to show that there is at least one `desirable' and at least $R-1$ `almost desirable' elements in $A'$. Then renaming desirable element as $a_1$ and almost desirable elements as $a_2, \ldots, a_R$ in equation (2) yields the result. Towards this, we take $T$ as complement of $A'$ in $[0,s+R-3]$ and $K=W \cup T$. The cardinality of $K$ is $k=w+R-2$. Let $d_1$ be the number of elements of $K$ which are smaller than $u_{k+1}$ and $d_2$ be the number of elements of $K$ which are bigger than $u_{k+1}$. We first assume that none of $d_i$ is zero and finish the proof.\\
Claim: When none of $d_i$ is zero then every element of $U$ in the interval $(u_{d_1},u_{u-d_2})$ is `desirable' and $u_{d_i}$ is `almost desirable'.\\
Let $u_j \in U \cap (u_{d_1},u_{u-d_2})$ and $w \in W$ then we wish to show that $u_j+w=u+v$ for some $u \in U, v \in U \cup V$.\\
case (1)- $u_{c+1} \leq u_j+w < u_{k+1}$.\\
All elements $u_j+w-u_r, 1 \leq r \leq c+1$ are smaller than $u_{k+1}$ and are in $[0,s+R-3]$, hence can not be in $K$, proving that $u_j$ is `desirable'.\\
case (2)- $u_{k+1} \leq u_j+w \leq s+R-3$.\\
Here, the elements $u_j+w-u_r, 1 \leq r \leq k+1$ can not all lie in $K$, making $u_j$ a `desirable' element.\\
case (3)- $u_j+w \leq s+R-3+u_{u-k}$.\\
In this case one of the elements from $u_j+w-u_r, u-k \leq r \leq u$ makes $u_j$ `desirable'.\\
case (4)- $u_j+w > s+R-3+u_{u-k}$.\\
Now one of the elements $u_j+w-u_r, u-d_2 \leq r \leq u$ assures that $u_j$ is `desirable'.\\
This produces $R-1$ desirable elements. Similar analysis shows that the elements $u_{d_1},u_{d_2}$ are `almost desirable'. In case one of $d_1 \mbox{ and }d_2$ is zero, say $d_2=0$. In this case clearly $u_j \mbox{ with }u_j+w \geq  u_{k+1} $ is `desirable' and there are $R-1$ such $u_j$, namely $u_r,
\mbox{ for }r\geq k+1$. We claim that $u_k$ is almost desirable. For this we notice that there can be at most one $w$ such that $\{u_k+w-u_j:1 \leq j \leq k \}=K$ and thus $u_k$ contributes to all but one good coordinate, proving that $u_k$ is almost desirable. The case when $d_1=0$ is similar.

\end{document}